\documentclass[12pt]{amsart}

\usepackage[a4paper]{geometry}

\usepackage{amsmath}
\usepackage{amsthm}
\usepackage{amsfonts}
\usepackage{amssymb}
\usepackage{mathabx}

\usepackage{graphicx}
\usepackage{epstopdf}
\usepackage{caption}
\usepackage{subcaption}

\usepackage[nodayofweek]{datetime}
\usepackage{hyperref}
\usepackage{hyphenat}

\newtheorem{theorem}{Theorem}[section]
\newtheorem{lemma}[theorem]{Lemma}
\newtheorem{proposition}[theorem]{Proposition}

\theoremstyle{definition}

\newtheorem{claim}[theorem]{Claim}
\newtheorem{remark}[theorem]{Remark}

\newcommand{\C}{\mathcal{C}}
\newcommand{\D}{\mathcal{D}}

\title{Uniform quasiconvexity of the disc graphs in the curve graphs}
\author{Kate M. Vokes}
\address{Mathematics Institute, University of Warwick, Coventry, CV4 7AL, United Kingdom}
\email{k.m.vokes@warwick.ac.uk}

\thanks{Date: \today}

\begin{document}

\maketitle

\begin{abstract}
We give a proof that there exists a universal constant $K$ such that the disc graph associated to a surface $S$ forming a boundary component of a compact, orientable 3-manifold $M$ is $K$-quasiconvex in the curve graph of $S$.
Our proof does not require the use of train tracks.
\end{abstract}

\section{Introduction}

Given a closed, connected, orientable surface $S$, the associated \emph{curve graph}, $\C(S)$, has as vertex set the set of isotopy classes of essential
simple closed curves in $S$, with an edge between two distinct vertices if the corresponding isotopy classes have representatives that are disjoint.
The definition is due to Harvey \cite{harvey}.
The curve graph has been a significant tool in the study of mapping class groups, Teichm\"uller spaces and hyperbolic 3\hyp{}manifolds.
If $S$ has genus at least 2, $\C(S)$ is connected, and Masur and Minsky showed in \cite{mm1} that it has infinite diameter and is hyperbolic in the sense of Gromov.
More recently, it was shown in independent proofs by Aougab \cite{aouhyp}, Bowditch \cite{bowhyp}, Clay, Rafi and Schleimer \cite{crs} and Hensel, Przytycki and Webb \cite{hpw} that the constant of hyperbolicity can be chosen to be independent of the surface $S$.

When $S$ is a boundary component of a compact, orientable 3-manifold $M$, we can consider the subset of the vertex set of $\C(S)$ which consists of those curves which bound embedded discs in $M$.
Equivalently, by Dehn's lemma, these are the essential simple closed curves in $S$ which are homotopically trivial in $M$.
The \emph{disc graph}, $\D(M, S)$, is the full subgraph spanned by these vertices.
This has applications to the study of handlebody groups and Heegaard splittings.
In a further paper of Masur and Minsky \cite{mm3}, it was proved that $\D(M, S)$ is $K$\hyp{}quasiconvex in $\C(S)$ (see Section 2 for a definition), for some $K$ depending only on the genus of $S$.
The proof relies on a study of nested train track sequences.
More specifically, to any pair of vertices of $\D(M, S)$, Masur and Minsky associate a sequence of curves in $\D(M, S)$, and a nested train track sequence whose vertex cycles are close in $\C(S)$ to the curves of this sequence.
They prove that the sets of vertex cycles of nested train track sequences are quasiconvex in $\C(S)$, and the result follows.

This result was improved by Aougab, who showed in \cite{aou} that the constants of quasiconvexity for nested train track sequences can be taken to be quadratic in the complexity of the surface, obtaining as a corollary that there exists a function $K(g)=O(g^2)$ such that $\D(M, S)$ is $K(g)$\hyp{}quasiconvex in $\C(S)$, where $g$ is the genus of $S$.
That this bound can be taken to be uniform in the genus of $S$ follows from work of Hamenst\"adt \cite{hamdiscs}.
In Section 3 of \cite{hamdiscs}, it is shown that the sets of vertex cycles of train track splitting sequences give unparametrised quasigeodesics in $\C(S)$ with constants independent of the surface $S$.
Along with the uniform hyperbolicity of the curve graphs, this implies that such subsets are uniformly quasiconvex in $\C(S)$.
In this note, we give a direct proof of the uniform quasiconvexity of $\D(M, S)$ in $\C(S)$, without using train tracks.

\begin{theorem} \label{uniform}
There exists $K$ such that, for any compact, orientable 3-manifold $M$ and boundary component $S$ of $M$, the disc graph, $\D(M, S)$, is $K$-quasiconvex in $\C(S)$.
\end{theorem}

For the main case, where the genus of $S$ is at least 2, this uses an observation that the disc surgeries of \cite{mm3} give a path of ``bicorn curves'', as described by Przytycki and Sisto in \cite{przs}.
These were introduced by analogy with the ``unicorn arcs'' of \cite{hpw} in order to give a short surgery proof of the uniform hyperbolicity of the curve graphs.
The lower genus case is straight-forward, and is discussed in Section \ref{low genus}.

\subsubsection*{Acknowledgements.} I am grateful to my supervisor, Brian Bowditch, for many helpful discussions and suggestions.
I would also like to thank Francesca Iezzi and Saul Schleimer for interesting conversations.
This work was supported by an Engineering and Physical Sciences Research Council Doctoral Award.

\section{Preliminaries} \label{prelim}

A simple closed curve in a surface $S$ is said to be \emph{essential} if it does not bound a disc in $S$.
We will assume from now on that all curves are essential simple closed curves.
Two curves, $\alpha$ and $\beta$, intersecting transversely, are in \emph{minimal position} if the number of their intersection points is minimal over all pairs $\alpha'$, $\beta'$ isotopic to $\alpha$, $\beta$ respectively.
This is equivalent to the condition that $\alpha$ and $\beta$ do not form a bigon, that is, an embedded disc in $S$ whose boundary is a union of one arc of each of $\alpha$ and $\beta$.
Abusing notation, we shall also denote the isotopy class of a curve $\alpha$ by $\alpha$.
The intersection number, $i(\alpha, \beta)$, is the number of intersections between representatives of the isotopy classes of $\alpha$ and $\beta$ which are in minimal position.

If $Y$ is a subset of a geodesic metric space $X$, we denote the closed $K$\hyp{}neighbourhood of $Y$ in $X$ by $N_X(Y, K)$.
We say that $Y$ is \emph{$K$-quasiconvex} in $X$ if, for any two points $y$ and $y'$ in $Y$, any geodesic in $X$ joining $y$ and $y'$ is contained within $N_X(Y, K)$.
The metric on the curve graph, $\C(S)$, is given by setting each edge to have length 1, and we denote the distance between vertices $\alpha$ and $\beta$ by $d_S(\alpha, \beta)$.
This makes $\C(S)$ into a geodesic metric space.

Given two curves $\alpha$ and $\beta$ in minimal position, a \emph{bicorn curve} between $\alpha$ and $\beta$ is a curve constructed from one arc $a$ of $\alpha$ and one arc $b$ of $\beta$ such that $a$ and $b$ meet precisely at their endpoints (see Figure \ref{bicorns} and \cite{przs}).
Since $\alpha$ and $\beta$ are in minimal position, all bicorn curves between them will be essential, as otherwise $\alpha$ and $\beta$ would form a bigon.

\begin{figure}[h!]
\begin{subfigure}[b]{0.35\textwidth}
\centering
\includegraphics[width=\textwidth]{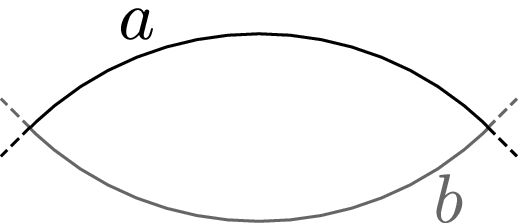}
\end{subfigure}
\qquad
\begin{subfigure}[b]{0.35\textwidth}
\centering
\includegraphics[width=\textwidth]{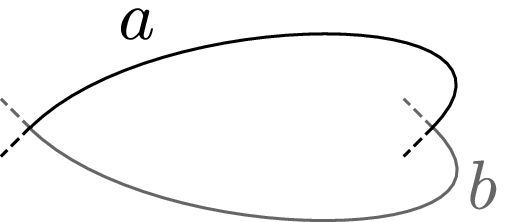}
\end{subfigure}
\caption{Examples of bicorn curves.}
\label{bicorns}
\end{figure}

\section{Exceptional cases} \label{low genus}

Let $M$ be a compact, orientable 3\hyp{}manifold.
If a boundary component $S$ has genus at most one then the associated disc graph, $\D(M, S)$, is very simple.
Firstly, since there are no essential simple closed curves on the sphere, the curve graph of the sphere is empty, so we can ignore any sphere boundary components.
We shall see that for a torus boundary component $S$, $\D(M, S)$ contains at most one vertex.

The curve graph of the torus contains infinitely many vertices, but, with the definition given above, is not connected, since any non-isotopic curves must intersect.
We instead use a standard modification and define two vertices to be adjacent if they have representative curves which intersect exactly once.
The resulting graph is the \emph{Farey graph}, which is connected.

Suppose $S$ is a torus boundary component of the 3\hyp{}manifold $M$.
Suppose an essential curve $\delta$ in $S$ bounds an embedded disc $D$ in $M$.
Take a closed regular neighbourhood $N$ of $S \cup D$ in $M$.
This is homeomorphic to a solid torus with an open ball removed.
Suppose some other curve $\delta'$ in $S$ bounds an embedded disc $D'$ in $M$.
We can assume that $D'$ intersects the sphere boundary component of $N$ transversely in simple closed curves.
Repeatly performing surgeries along innermost discs to reduce the number of such curves eventually gives a disc with boundary $\delta'$ which is completely contained in $N$.
Therefore, an essential curve in $S$ bounds an embedded disc in $M$ if and only if it bounds an embedded disc in $N$.
In $S$, there is, up to isotopy, no curve other than $\delta$ which bounds an embedded disc in $N$, since such a curve must be trivial in the first homology group of $N$.
We hence find that for any torus boundary component $S$, $\D(M, S)$ is at most a single point.
In this case, $\D(M, S)$ is convex in the curve graph of $S$.

\section{Proof of the main result}

Now let $S$ be a boundary component of genus at least 2 of a compact, orientable $3$\hyp{}manifold $M$, and $\mathcal{D}(M, S)$ the associated disc graph.

The following criterion for hyperbolicity appears in several places in slightly different forms, for example as Theorem 3.15 of \cite{ms}.
For our purposes, the important result is the final clause on Hausdorff distances, which appears in the statement of Proposition 3.1 of \cite{bowhyp}.

This criterion is used in \cite{przs} to show that the curve graphs of closed surfaces of genus at least 2 are uniformly hyperbolic.

\begin{proposition} \label{criterion}
For all $h \ge 0$, there exist $k$ and $R$ such that the following holds.
Let $G$ be a connected graph with vertex set $V(G)$.
Suppose that for every \mbox{$x, y \in V(G)$} there is a connected subgraph \mbox{$\mathcal{L}(x, y) \subseteq G$,} containing $x$ and $y$, with the following properties:
\begin{enumerate}
\item for any $x, y \in V(G)$ with $d_G(x, y) \le 1$, the diameter of $\mathcal{L}(x, y)$ in $G$ is at most $h$,
\item for all $x, y, z \in V(G)$, $\mathcal{L}(x, y) \subseteq N_G(\mathcal{L}(x, z) \cup \mathcal{L}(y, z), h)$.
\end{enumerate}
Then $G$ is $k$-hyperbolic.
Furthermore, for all $x, y \in V(G)$, the Hausdorff distance between $\mathcal{L}(x, y)$ and any geodesic from $x$ to $y$ is at most $R$.
\end{proposition}

We first note that a slightly more general statement is true.

\begin{claim} \label{coarsely connected}
The requirement that each $\mathcal{L}(x, y)$ be connected can be replaced by the weaker condition that there exist $h'$ such that, for all $x, y \in V(G)$, $N_G(\mathcal{L}(x, y), h')$ is connected.
\end{claim}

To see this, suppose subgraphs $\mathcal{L}(x, y)$ satisfy all the hypotheses of Proposition \ref{criterion}, except that the connectedness assumption is replaced as described in the claim.
Define $\mathcal{L}'(x, y) = N_G(\mathcal{L}(x, y), h')$.
This is a connected subgraph of $G$ containing $x$ and $y$.
For any $x, y \in V(G)$ with $d_G(x, y) \le 1$, the diameter of $\mathcal{L}'(x, y)$ in $G$ is at most $h+2h'$, and for any $x, y, z \in V(G)$, $\mathcal{L}'(x, y) \subseteq N_G(\mathcal{L}'(x, z) \nolinebreak \cup \nolinebreak \mathcal{L}'(y, z), \nolinebreak h)$.
Hence, the conclusion of Proposition \ref{criterion} holds, except with constants now depending on $h$ and $h'$, proving Claim \ref{coarsely connected}.

Given two curves $\alpha$ and $\beta$ in $S$, we shall define $\Theta(\alpha, \beta)$ to be the set containing the isotopy classes of $\alpha$, $\beta$ and all bicorn curves between $\alpha$ and $\beta$.

Przytycki and Sisto define in \cite{przs} an ``augmented curve graph'', $\C_{aug}(S)$, where two curves are adjacent if they intersect at most twice.
Such curves cannot fill $S$ (which has genus at least 2) so are at distance at most 2 in $\C(S)$.
Given two curves $\alpha$ and $\beta$ in minimal position, $\eta(\alpha, \beta)$ is defined in \cite{przs} to be the full subgraph of $\mathcal{C}_{aug}(S)$ spanned by $\Theta(\alpha, \beta)$.
This is shown to be connected for all $\alpha$ and $\beta$.
It is further verified that the hypotheses of Proposition \ref{criterion} are satisfied when $G$ is $\C_{aug}(S)$, $\mathcal{L}(\alpha, \beta)$ is $\eta(\alpha, \beta)$ for each $\alpha$, $\beta$, and $h$ is 1, independently of the surface \nolinebreak $S$.

Since $\eta(\alpha, \beta)$ is connected, for any $\gamma, \gamma' \in \Theta(\alpha, \beta)$, there is a sequence $\gamma=\gamma_0, \gamma_1, \dots, \gamma_n=\gamma'$ of curves in $\Theta(\alpha, \beta)$, where $d_S(\gamma_{i-1}, \gamma_i) \le 2$ for each $1 \nolinebreak \le \nolinebreak i \le \nolinebreak n$.
Hence, $N_{\C(S)}(\Theta(\alpha, \beta), 1)$ is a connected subgraph of $\C(S)$.
Moreover, if $d_S(\alpha, \beta) \le 1$, then $\alpha$ and $\beta$ are disjoint, so $\Theta(\alpha, \beta)$ contains no other curves and its diameter in $\C(S)$ is at most 1.
Finally, since $\eta(\alpha, \beta) \subset N_{\C_{aug}(S)}(\eta(\alpha, \delta) \cup \eta(\beta, \delta), 1)$ for any curves $\alpha$, $\beta$, $\delta$, we have $\Theta(\alpha, \beta) \subset N_{\C(S)}(\Theta(\alpha, \delta) \cup \Theta(\beta, \delta), 2)$.

Using Proposition \ref{criterion} with the modification of Claim \ref{coarsely connected}, this proves the following lemma.

\begin{lemma} \label{hausdorff distance}
There exists $R$ such that, for any closed surface $S$ of genus at least 2, and any curves $\alpha$, $\beta$ in $S$, the Hausdorff distance in $\C(S)$ between $\Theta(\alpha, \beta)$ and any geodesic in $\C(S)$ joining $\alpha$ and $\beta$ is at most $R$.
\end{lemma}

We now show that, moreover, any geodesic between $\alpha$ and $\beta$ in $\C(S)$ lies in a uniform neighbourhood of any path within $\Theta(\alpha, \beta)$ connecting $\alpha$ and $\beta$.

\begin{lemma} \label{claim3}
Let $\alpha$, $\beta$ be two curves in $S$, $P(\alpha, \beta)$ a path from $\alpha$ to $\beta$ in $\mathcal{C}(S)$ with all vertices in $\Theta(\alpha, \beta)$, and $g$ a geodesic in $\C(S)$ joining $\alpha$ and $\beta$.
Then $g$ is contained in the $(2R+2)$\hyp{}neighbourhood of $P(\alpha, \beta)$.
\end{lemma}

\begin{proof}
This uses a well known connectedness argument.
From Lemma \ref{hausdorff distance}, $P(\alpha, \beta)$ is contained in $N_{\C(S)}(g, R)$.
Take any vertex $\gamma$ in $g$.
Let $g_0$ be the subpath of $g$ from $\alpha$ to $\gamma$ and $g_1$ the subpath from $\gamma$ to $\beta$.
Then the three sets $N_{\C(S)}(g_0, R+1)$, $N_{\C(S)}(g_1, R+1)$ and $P(\alpha, \beta)$ intersect in at least one vertex, say $\delta$.
Let $\gamma_0$ in $g_0$ and $\gamma_1$ in $g_1$ be such that $d_S(\gamma_0, \delta) \le R+1$ and $d_S(\gamma_1, \delta) \le R+1$.
Now $d_S(\gamma_0, \gamma_1) \le 2R+2$ and $\gamma$ is in the (geodesic) subpath of $g$ from $\gamma_0$ to $\gamma_1$, so $d_S(\gamma, \gamma_i) \le R+1$ for either $i=0$ or $i=1$.
Hence, $d_S(\gamma, \delta) \le 2R+2$.
Since $\gamma$ was an arbitrary vertex in $g$ and $\delta$ is in $P(\alpha, \beta)$, we have $g \subset N_{\C(S)}(P(\alpha, \beta), 2R+2)$.
\end{proof}

Given that $\alpha$ and $\beta$ bound embedded discs in $M$, we now describe how to choose $P(\alpha, \beta)$ so that all curves in the path are also vertices of $\mathcal{D}(M, S)$, following Section 2 of \cite{mm3}.

Assume curves $\alpha$ and $\beta$ are fixed in minimal position and choose a subarc $J \subset \alpha$ containing all points of $\alpha \cap \beta$ in its interior.
Masur and Minsky define several curve replacements, of which we shall need only the following.
A \emph{wave curve replacement} with respect to $(\alpha, \beta, J)$ is the replacement of $\alpha$ and $J$ by $\alpha_1$ and $J_1$ as follows (see Figure \ref{wave}).
Let $w$ be a subarc of $\beta$ with interior disjoint from $\alpha$, and endpoints $p$, $q$ in $J$.
Suppose that $w$ meets the same side of $J$ at both $p$ and $q$; then $w$ is called a \emph{wave}.
Let $J_1$ be the subarc of $J$ with endpoints $p$, $q$.
Define $\alpha_1$ to be the curve $w \cup J_1$.
This is an essential curve since $\alpha$ and $\beta$ are in minimal position, so, in particular, no subarc of $J$ and subarc of $\beta$ can form a bigon.
Where $\text{int}(J) \cap \beta = \emptyset$, we define a curve replacement with respect to $(\alpha, \beta, J)$  by $\alpha_1=\beta$, $J_1=J$.

\begin{figure}[h!]
\centering
\includegraphics[width=0.55\textwidth]{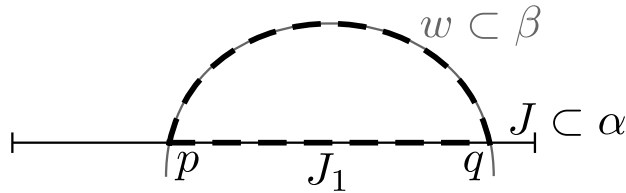}
\caption{A wave curve replacement. The dashed curve is $\alpha_1$.}
\label{wave}
\end{figure}

\begin{remark}
In \cite{mm3}, it is arranged that $\alpha_1$ and $\beta$ must intersect transversely and be in minimal position by requiring an additional condition on the wave $w$ and by slightly isotoping $w \cup J_1$ to be disjoint from $w$.
However, this will not be necessary here, so we choose to simplify the exposition by removing this condition.
\end{remark}

Notice that since $\alpha$ does not intersect $\text{int}(w)$, $i(\alpha, \alpha_1)=0$.
Moreover, $\alpha_1 \cap \beta$ consists of the arc $w$ and a set of points which are all contained in the interior of $J_1$, and $\lvert \beta \cap \text{int}(J_1) \rvert < \lvert \beta \cap \text{int}(J) \rvert$.

As $\alpha$ and $\beta$ are in minimal position, no subarc of $J_1$ can form a bigon with a subarc of $\beta$, so this process can be iterated.
A \emph{nested curve replacement sequence} is a sequence $\lbrace (\alpha_i, J_i) \rbrace$ of curves $\alpha=\alpha_0, \alpha_1, \dots, \alpha_n$ and subarcs $\alpha \supset J_0 \supset J_1 \supset \dots \supset \nolinebreak J_n$, such that $J_0$ contains all points of $\alpha \cap \beta$, and such that $\alpha_{i+1}$ and $J_{i+1}$ are obtained by a curve replacement with respect to $(\alpha_i, \beta, J_i)$.
We always have $i(\alpha_i, \alpha_{i+1})=0$, as for $\alpha$ and $\alpha_1$.
Observe that all curves $\alpha_i$ in this sequence are bicorn curves between $\alpha$ and $\beta$, since the nested intervals ensure that they are formed from exactly one arc of $\alpha$ and one of $\beta$.

The following is a case of Proposition 2.1 of \cite{mm3}.
We include a proof for completeness, with the minor modification of the slightly different curve replacements.

\begin{proposition} \label{discprop}
Let $S$ be a boundary component of a compact, orientable 3-manifold $M$, and let $\alpha$ and $\beta$ be two curves in $S$ in minimal position, each of which bounds an embedded disc in $M$.
Let $J_0 \subset \alpha$ be a subarc containing all points of $\alpha \cap \beta$.
Then there exists a nested curve replacement sequence $\lbrace (\alpha_i, J_i) \rbrace$, with $\alpha_0=\alpha$, such that:
\begin{itemize}
\item each $\alpha_i$ bounds an embedded disc in $M$,
\item the sequence terminates with $\alpha_n=\beta$.
\end{itemize}
\end{proposition}

\begin{proof}
Suppose that $\alpha$ and $\beta$ bound properly embedded discs $A$ and $B$ respectively.
We can assume that $A$ and $B$ intersect transversely, so their intersection locus is a collection of properly embedded arcs and simple closed curves.
Furthermore, we can remove any simple closed curve components by repeatedly performing surgeries along innermost discs, so that $A$ and $B$ intersect only in properly embedded arcs.
We will perform surgeries on these discs to get a sequence of discs $A_i$ with $\partial A_i = \alpha_i$.
Throughout the surgeries, we will keep $A$ and $B$ fixed, and each $A_i$, except $A_0=A$ and $A_n=B$, will be a union of exactly one subdisc of each of $A$ and $B$.

Suppose the sequence is constructed up to $\alpha_i = \partial A_i$.
If $\beta \cap \alpha_i$ is empty, then $\alpha_{i+1}=\beta = \partial B$ by definition, so the sequence is finished.

\begin{figure}[h!]
\centering
\includegraphics[width=0.5\textwidth]{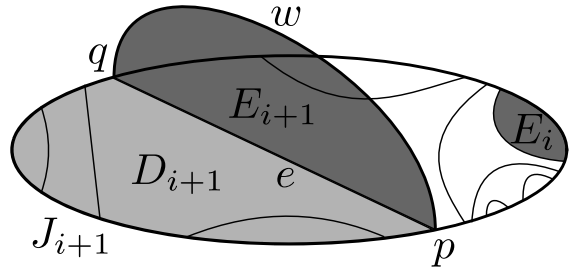}
\caption{The disc surgeries of Proposition \ref{discprop}. The horizontal disc is $A_i$, shown with arcs of intersection with $B$.}
\label{discsurgery}
\end{figure}

Suppose $\beta$ intersects $\alpha_i$ (as illustrated in the example of Figure \ref{discsurgery}).
Let $A_i = D_i \cup E_i$, where $D_i$ is a subdisc of $A$ and $E_i$ is a subdisc of $B$.
If $i=0$, then $E_i$ is empty.
If $i>0$, let $J_i$ be the arc of $\partial D_i$ which is contained in $\partial A_i$.
Any point of intersection of $\beta$ and $J_i$ is an endpoint of an arc of intersection of $B$ and $D_i$.
Let $E_{i+1}$ be a disc in $B$ such that the boundary of $E_{i+1}$ is made up of an arc $e$ in $\text{int}(D_i) \cap B$ and a subarc $w$ of $\beta$, and such that the interior of $E_{i+1}$ is disjoint from $A_i$ (that is, $E_{i+1}$ is an outermost component of $B \setminus (A_i \cap B)$).
This in particular means that the interior of $w$ is disjoint from $\alpha_i$, that the endpoints $p$, $q$ of $w$ lie in the interior of $J_i$, and that $w$ meets the same side of $J_i$ at both of these endpoints, so $w$ is a wave.
Let $J_{i+1}$ be the subarc of $J_i$ with endpoints $p$, $q$.
Let $D_{i+1}$ be the disc in $A_i$ bounded by $e \cup J_{i+1}$.
This disc is contained in $D_i$ and hence in $A$.
The curve $w \cup J_{i+1}$, with interval $J_{i+1}$, is the wave curve replacement $\alpha_{i+1}$ obtained from $(\alpha_i, \beta, J_i)$, and it is also the boundary of the embedded disc $A_{i+1}=D_{i+1} \cup E_{i+1}$.

Since at each stage $\lvert \beta \cap \text{int}(J_i) \rvert$ decreases, this terminates with $\lvert \beta \cap \text{int}(J_{n-1}) \rvert=0$ and \mbox{$\alpha_n=\beta$.} 
\end{proof}

This sequence defines the vertices of a path $P(\alpha, \beta)$ in $\C(S)$, with these vertices contained in both $\mathcal{D}(M, S)$ and $\Theta(\alpha, \beta)$.
By Lemma \ref{claim3}, there exists $K$, independent of $S$, $\alpha$ and $\beta$, such that any geodesic $g$ joining $\alpha$ and $\beta$ in $\C(S)$ is contained in the closed $K$\hyp{}neighbourhood of $P(\alpha, \beta)$.
Hence, $g$ is contained in the closed $K$\hyp{}neighbourhood of $\D(M, S)$, completing the proof of Theorem \ref{uniform}.

\end{document}